\documentclass[a4paper, 11pt]{amsart}
\makeatletter
\@namedef{subjclassname@2020}{%
	\textup{2020} Mathematics Subject Classification}
\makeatother

\usepackage{amsmath,amsthm,verbatim,amssymb,amsfonts,mathrsfs,amscd, graphicx,enumerate,tikz-cd,stmaryrd}
\usetikzlibrary{trees}
\usepackage{multicol, caption}
\usepackage{fullpage}
\usepackage{enumitem}
\setlist[enumerate,1]{label={$(\roman*)$},leftmargin=*}

\usepackage{graphicx} 

\usepackage{amsmath}               
\usepackage{amsfonts}              
\usepackage{amsthm}                
\usepackage{mathtools}
\usepackage{verbatim}
\usepackage{xifthen}
\usepackage{amssymb}
\usepackage{xstring}
\usepackage{mathabx}
\usepackage{array}  
\usepackage{parskip} 
\usepackage{booktabs}
\usepackage[all]{xy}
\usepackage[pagebackref, hypertexnames=false]{hyperref} 
\hypersetup{
    colorlinks=true,
    linkcolor={red!50!black},
	citecolor={blue!50!black},
	urlcolor={blue!80!black}
}
\usepackage{tikz-cd}
\usepackage{stmaryrd} 
\usepackage{cleveref}
\usepackage{umoline}

\newtheorem{thm}{Theorem}[section]
\newtheorem{lem}[thm]{Lemma}
\newtheorem{prop}[thm]{Proposition}
\theoremstyle{definition}
\newtheorem{defn}[thm]{Definition}

\newtheorem*{exmp}{Example}
\newtheorem{cor}[thm]{Corollary}
\newtheorem{remark}[thm]{Remark}
\newtheorem{remarks}[thm]{Remarks}

\newtheorem*{question*}{Question}

\newtheorem{question}{Question}

\numberwithin{thm}{section}
\numberwithin{equation}{section}

\usepackage[msc-links,alphabetic]{amsrefs}

\newcommand{\dotifempty} [1]{\ifthenelse{\isempty{#1}}
                          	{\cdot}%
                          	{#1}}

\DeclarePairedDelimiter\floor{\lfloor}{\rfloor}
\DeclarePairedDelimiter\ceiling{\lceil}{\rceil}

\DeclarePairedDelimiter\braces{\{}{\}}

\newcommand{\sog}[1]{\left( #1 \right)}

\newcommand{\map}[0]{\rightarrow}

\renewcommand{\restriction}[2]{{
  \left.\kern-\nulldelimiterspace 
  #1 
  \vphantom{\big|} 
  \right|_{#2} 
  }}

\newcommand{\twomat}[4]{\begin{pmatrix}#1 & #2 \\ #3 & #4 \end{pmatrix}}

\DeclareMathOperator{\trace}{tr}

\newcommand{\fld}[1]{
						\ifthenelse{\isempty{#1}}
                          	{\mathbb{Q}_p}
                          	{#1}
                         } 
\renewcommand{\O}[1]{
						\ifthenelse{\isempty{#1}}
                          	{\mathcal{O}}
                          	{\mathcal{O}}
                         } 

\newcommand{\oneOkMesh}[1]{
						\ifthenelse{\isempty{#1}}
                          	{\mathcal{O}_{\fld{K}}}%
                          	{T^{#1}\cdot\mathcal{O}_{\fld{K}}}
                         }

\newcommand{\oneOk}[2]{
						\ifthenelse{\isempty{#1}}
                          	{\textbf{1}_{\oneOkMesh{#2}}}%
                          	{\textbf{1}_{#1 +\oneOkMesh{#2}}}
                         }

\newcommand{\oneZpMesh}[1]{
						\ifthenelse{\isempty{#1}}
                          	{\fld{Z}_p}%
                          	{p^{#1}\cdot\fld{Z}_p}
                         }

\newcommand{\oneZp}[2]{
						\ifthenelse{\isempty{#1}}
                          	{\textbf{1}_{\oneZpMesh{#2}}}%
                          	{\textbf{1}_{#1 +\oneZpMesh{#2}}}
                         }

\newcommand{\eps}[0]{\epsilon}

\graphicspath{ {images/} }

\usepackage{shortcuts}

\begin{document}
\title{On Ribet's lemma for $\mathrm{GL}_2$ modulo prime powers}
\author{Amit Ophir and Ariel Weiss}

\address{Amit Ophir, \newline
Department of Mathematics, University of California San Diego, La Jolla, CA, USA
\newline Einstein Institute of Mathematics, The Hebrew University of Jerusalem, Jerusalem, Israel.\vspace*{-3pt}}
\email{aophir@ucsd.edu}
\address{Ariel Weiss,\newline Department of Mathematics, The Ohio State University, Columbus, OH, USA\newline Department of Mathematics, Ben-Gurion University of the Negev, Be'er Sheva, Israel.\vspace*{-6pt}}
\email{weiss.742@osu.edu}

\begin{abstract}

 Let $\rho\:G\to \GL_2(K)$ be a continuous representation of a compact group $G$ over a complete discretely valued field $K$ with ring of integers $\O{K}$ and uniformiser $\pi$. We prove that $\trace\rho$ is reducible modulo $\pi^n$ if and only if $\rho$ is reducible modulo $\pi^n$. More precisely, there exist characters $\chi_1,\chi_2\:G\to(\O{K}/\pi^n\O{K})\t$ such that $\det(t - \rho(g))\equiv (t-\chi_1(g))(t-\chi_2(g))\pmod{\pi^n}$ for all $g\in  G$, if and only if there exists a $G$-stable lattice $\Lambda\sub K^2$ such that $\Lambda/\pi^n\Lambda$ contains a $G$-invariant, free, rank one $\O{K}/\pi^n\O{K}$-submodule. Our result applies in the case that $\rho$ is not residually multiplicity free, in which case it answers a question of Bella\"iche--Chenevier \cite{Bellaiche-arbres}*{Question, pp.\ 524}. As an application, we prove an optimal version of Ribet's Lemma, which gives a condition for the existence of a $G$-stable lattice $\Lambda$ that realises a non-split extension of $\chi_2$ by $\chi_1$.

\end{abstract}

\maketitle

\section{Introduction}

Let $\O{K}$ be a complete discrete valuation ring with fraction field $K$, uniformiser $\pi$, discrete valuation $v_\pi$ normalised such that $v_\pi(\pi)=1$ and residue field $\F$. Let $G$ be a compact group and let $\rho\:G\to \GL_2(K)$ be a continuous representation. Suppose that there exists an integer $n$ and continuous characters $\chi_1, \chi_2\:G\to (\O{K}/\pi^n\O{K})\t$ such that, for all $g\in G$, we have
\begin{equation}\label{eq:det}
P_{\rho(g)}(t):=\det(t - \rho(g)) \equiv (t-\chi_1(g))(t-\chi_2(g))\pmod{\pi^n},    
\end{equation}
where $P_{\rho(g)}(t)$ is the characteristic polynomial of $\rho(g)$.
The goal of this paper is to answer the following questions, the first of which is equivalent to \cite{Bellaiche-arbres}*{Question, pp. 524}:

\begin{question}\label{question1}
Is $\rho$ reducible modulo $\pi^n$, i.e.\ does there exist a $G$-stable lattice $\Lambda\sub K^2$ such that $\Lambda/\pi^n\Lambda$ contains a $G$-invariant, free, rank one $\O{K}/\pi^n\O{K}$-submodule $V$?
\end{question}

 Equivalently, does there exist a basis for $K^2$ with respect to which the image of $\rho$ is a subgroup of 
\[\Gamma_0(\pi^n) := \set{\begin{pmatrix}a&b\\c&d\end{pmatrix}\in \GL_2(\O{K}) : c\equiv 0\pmod{\pi^n}}?\]

\begin{question}\label{question2}
Can we choose $\Lambda$ so that the rank one submodule $V$ is isomorphic to $\chi_1$? Moreover, if $\rho$ is irreducible, can we choose $\Lambda$ so that $\Lambda/\pi^n\Lambda$ is a non-split extension of $\chi_2$ by $\chi_1$?
\end{question}

\Cref{question1} has previously been studied by Katz \cite{KatzAbelian1980} in the context of Galois representations attached to elliptic curves, and answered in the case that $K$ is a finite extension of $\Qp$ and $\chi_1$ is the trivial character \cite{KatzAbelian1980}*{Thm.\ 1}, though his proof can be generalised to any $\chi_1,\chi_2$. Our proof is completely different and works in the more general case of discretely valued fields. Moreover, our argument only requires $G$ to be a semigroup.

\subsection{The Bruhat--Tits tree of \texorpdfstring{$\PGL_2(K)$}{PGL(2, K)}}

In order to answer these questions, we rephrase them in the language of Bruhat--Tits trees.

By a lattice in $K^2$ we mean a rank two $\O{K}$-module $\Lambda\sub K^2$ that spans $K^2$ as a vector space. We say that two lattices $\Lambda_1, \Lambda_2$ are homothetic if there is an element $a\in K$ such that $\Lambda_1 = a\Lambda_2$.
\begin{defn}\label{def:bt-tree}
    The \emph{Bruhat--Tits tree} $\X$ of $\PGL_2(\fld{K})$ is the graph whose vertices are the homothety classes of lattices in $\fld{K}^2$. Two vertices $x, y\in \X$ are joined by an edge if we can choose representatives $\Lambda_x,\Lambda_y$ of $x, y$ such that
\[\pi\Lambda_x\subsetneq \Lambda_y\subsetneq \Lambda_x.\]
\end{defn}
Equivalently, $x, y\in \X$ are neighbours if there exists a basis $(v_1, v_2)$ of $\Lambda_x$ such that the homothety class of the lattice with basis $(v_1, \pi v_2)$ is $y$. After fixing such a basis, the other neighbours of $\Lambda_x$ are the lattices with bases $(\pi v_1, v_2 + i v_1)$ where $i\in \O{K}$ runs over a set of representatives for the congruence classes of $\O{K}$ modulo $\pi$.
Thus, if $q$ denotes the (possibly infinite) cardinality of $\F$, then $\X$ is a $(q+1)$-regular tree.

By extension, suppose that $x, y\in \X$ are two vertices of $\X$ of distance $d = d(x, y)$ from each other, where $d(x,y)$ is the number of edges in the path connecting $x$ to $y$. Then $d$ is the smallest integer for which we can choose representatives $\Lambda_x,\Lambda_y$ such that
\[\pi^d\Lambda_x\subsetneq \Lambda_y\subsetneq \Lambda_x.\]
Equivalently, $d$ is the unique integer for which there exists a basis $(v_1, v_2)$ of $\Lambda_x$ such that $(v_1, \pi^d v_2)$ is a basis of $\Lambda_y$.

The representation $\rho$ induces an action of $G$ on the vertices of $\X$. Let $\X(\rho)$ be the subgraph of $\X$ whose vertices are homothety classes of $G$-stable lattices. Then, for each $x\in\X(\rho)$ with representative $\Lambda_x$, there is a bijection between vertices $y\in\X(\rho)$ with $d(x, y) = n$ and free, rank one, $G$-invariant $\O{K}/\pi^n\O{K}$-submodules of $\Lambda_x/\pi^n\Lambda_x$ (\Cref{lem:distance-submod-bijection}). Thus, an affirmative answer to \Cref{question1} follows from the following theorem, which is our main result:
\begin{thm}\label{thm:main-intro}
The following are equivalent:
        \begin{enumerate}
        \item There exist two vertices $x,y\in \X(\rho)$ with $d(x,y)=n$.
                
        \item There exists a pair of characters $\chi_1, \chi_2\:G\to(\O{K}/\pi^n\O{K})\t$ such that, for all $g\in G$, $P_{\rho(g)}(t)\equiv (t- \chi_1(g))(t- \chi_2(g))\pmod{\pi^n}$.
        \end{enumerate}
\end{thm}
Questions \ref{question1} and \ref{question2} have been answered in far greater generality in the case that $\rho$ is residually multiplicity-free, i.e.\ if $\chi_1\not\equiv\chi_2\pmod\pi$ \cites{Bellaiche-apropos,Bellaiche-Graftieaux,Bellaiche-book, Chenevier-determinants}. In the residually multiplicity-free case, the graph $\X(\rho)$ is just a finite line segment, and the theory is simplified considerably by the fact that the characters $\chi_1, \chi_2$, if they exist, are unique. 

However, the non-residually multiplicity-free case is far more mysterious. This case has been studied by Bella\"iche--Chenevier \cite{Bellaiche-arbres}, who classify the types of graphs $\X(\rho)$ can be, prove that $(i)$ implies $(ii)$ in \Cref{thm:main-intro} and give a partial result in the opposite direction \cite{Bellaiche-arbres}*{Thm.\ 45}. They pose as a question whether $(i)$ and $(ii)$ are equivalent \cite{Bellaiche-arbres}*{Question, pp.\ 524}. As well as answering this question, we complete Bella\"iche--Chenevier's description of the shape of $\X(\rho)$: we describe the shape of $\X(\rho)$ in terms of certain pseudocharacter invariants of $\rho$ (\Cref{thm:classification1}). Conversely, we also show how to compute these invariants from the shape of $\X(\rho)$ (\Cref{thm_classification}).

\subsection{Ribet's Lemma}

In his celebrated 1976 paper \cite{ribet1976modular}, Ribet pioneered a technique to construct non-split extensions of Galois representations using congruences between modular forms. Suppose that $\rho\:\Ga\Q\to \GL_2(K)$ is an irreducible Galois representation that is residually reducible: there exist characters $\chi_1, \chi_2\:\Ga\Q\to \F\t$ such that $P_{\rho(g)}(t)\equiv (t- \chi_1(g))(t- \chi_2(g))\pmod{\pi}$. Ribet showed that there exists a $\Ga\Q$-stable lattice $\Lambda$ such that $\Lambda/\pi\Lambda$ is a non-split extension of $\chi_2$ by $\chi_1$. 

Ribet's Lemma has been generalised to higher-dimensional representations, to mod $\pi^n$ congruences, and to representations over more general rings \cites{Urban-ribet-lemma, Bellaiche-apropos,Bellaiche-Graftieaux,Brown, Bellaiche-book}. These generalisations are crucial components of proofs of Iwasawa main conjectures and cases of the Bloch--Kato conjecture \cites{Wiles_Iwasawa, Urban,Bellaiche-book, Skinner-Urban}. However, these generalisations all assume that $\rho$ is residually multiplicity-free. 

The non-residually multiplicity-free case is, again, more mysterious. Indeed, the na\"ive generalisation of Ribet's lemma is false: in \Cref{sec:counterexample}, we give examples of representations $\rho\:G\to \GL_2(K)$ with $P_{\rho(g)}(t)\equiv (t- \chi_1(g))(t- \chi_2(g))\pmod{\pi^n}$, but such that there is no $G$-stable lattice $\Lambda\sub K^2$ for which either of $\chi_1, \chi_2$ is a submodule of $\Lambda/\pi^n\Lambda$. Note that, if $\rho$ is irreducible, then \Cref{thm:main-intro} shows that there do exist pairs of characters $\eta_1, \eta_2$ and $G$-stable lattices $\Lambda$ such that $\Lambda/\pi^n\Lambda$ is a non-split extension of $\eta_2$ by $\eta_1$. The problem is that, unlike in the multiplicity-free case, the decomposition of $P_{\rho(g)}\pmod{\pi^n}$ as a product of characteristic polynomials of characters is not unique, and only some of the decompositions can actually be realised by lattices.  

Our first application of \Cref{thm:main-intro} is the following generalisation of Ribet's Lemma, which is optimal, in the sense that the integer $s$ is as large as possible in general:

\begin{thm}\label{thm:ribet-intro}
    Let $\rho\:G\to\GL_2(K)$ be an irreducible representation and let $\chi_1,\chi_2\:G\to (\O{K}/\pi^n\O{K})\t$ be characters such that, for all $g\in G$, the characteristic polynomial $P_{\rho(g)}(t)$ of $\rho(g)$ factors as  $(t- \chi_1(g))(t- \chi_2(g))\pmod{\pi^n}$. Let $m=m(\chi_1, \chi_2)$ be the largest integer such that $\chi_1\equiv \chi_2\pmod{\pi^m}$. Define
    \[s = s(\chi_1, \chi_2) = \begin{cases}n-m(\chi_1, \chi_2) &\text{if }m(\chi_1,\chi_2)<\frac n2\\ \ceiling*{\frac n2} &\text{if } m(\chi_1,\chi_2)\ge\frac n2.\end{cases}\]
    Then there exists a $G$-stable lattice $\Lambda$ such that $\Lambda/\pi^s\Lambda$ is a residually non-split extension of $\chi_1\pmod{\pi^s}$ by $\chi_2\pmod{\pi^s}$.
\end{thm}

Here, by a residually non-split extension, we mean that $\Lambda/\pi\Lambda$ is indecomposable.
We note in \Cref{cor:m-dichotomy} that the integer $s$ depends only on $\rho$ and $n$ and not on the choice of characters $\chi_1, \chi_2$.

\subsection{Isogenies of elliptic curves}

Let $F$ be a number field and let $E/F$ be an elliptic curve. If $\p$ is a prime of good reduction for $E$ and if the absolute ramification index $e_{\p}$ of $\p$ satisfies $e_\p<p-1$, then there is an injective map $E(F)\tors\hookrightarrow E(\F_{\p})$ from the torsion subgroup of $E(F)$ to the points of $E$ over the residue field $\F_{\p}$. In particular, if $\l$ is a prime, $n\ge 1$ and $\l^n\mid \#E(F)\tors$, then $\l^n\mid \#E(\F_{\p})$ for all primes $\p$ of good reduction.

In \cite{KatzAbelian1980}, Katz studied \Cref{question1} in order to prove a converse to this statement. Let \[\rho_{\l}\:\Ga F\to \GL_2(\Ql)\]
be the $\l$-adic Galois representation attached to (the isogeny class of) $E$. Then lattices inside $\rho_\l$ are in bijective correspondence with elliptic curves $E'$ that are $\l$-power isogenous to $E$, and the $\l$-isogeny graph of $E$ is exactly the invariant subtree $\X(\rho_\l)$. By the Chebotarev density theorem and the definition of $\rho_\l$, the condition that $\l^n\mid \#E(\F_{\p})$ for all primes $\p$ of good reduction is equivalent to $\rho_\l$ satisfying
\[\det(1 - \rho_\l(g))\equiv 0\pmod{\l^n}\]
for all $g\in \Ga F$, which is equivalent to $(\ref{eq:det})$ with $\chi_1= 1$ and $\chi_2$ the mod $\l^n$ cyclotomic character.
By answering \Cref{question1} with $K = \Qp$ and $\chi_1 = 1$, Katz showed that if $\l^n\mid \#E(\F_{\p})$ for all primes $\p$ of good reduction, then there is an elliptic curve $E'$, isogenous to $E$ over $F$, such that $\l^n\mid \#E'(F)\tors$. 

Similarly, if $E$ admits a cyclic $\l^n$-isogeny over $F$, then there is a character $\chi\:\Ga F\to (\Z/\l^n\Z)\t$ such that $\chi(g)$ is a root of the characteristic polynomial $P_{\rho(g)}(t)$ modulo $\l^n$, for all $g\in \Ga F$. Moreover, if $E$ has good reduction at a prime $\p\nmid\l$, then $\chi$ is unramified at $\p$. On the other hand, $E$ is isogenous over $F$ to an elliptic curve $E'$ that admits a cyclic $\l^n$-isogeny if and only if there is a line of distance $n$ in its $\l$-isogeny graph, which is exactly $\X(\rho_\l)$. Hence, as an immediate consequence of \Cref{thm:main-intro}, we deduce the following corollary:

\begin{cor}\label{thm:ell-curves}
Let $E$ be an elliptic curve over a number field $F$, and let $S$ be the set of primes of bad reduction for $E$ and the primes above a rational prime $\l$. Let $\rho_\l\:\Ga F\to \GL_2(\Ql)$ be the $\l$-adic Galois representation attached to $E$. The following are equivalent:
\begin{enumerate}
    \item $E$ is isogenous over $F$ to an elliptic curve $E'$ that admits a cyclic $\l^n$-isogeny.
    \item There exists a character $\chi\: \Ga F\to (\Z/\l^n\Z)\t$, unramified outside $S$, such that $\chi(\Frob_{\p})$ is a root of the characteristic polynomial $P_{\rho_\l(\Frob_{\p})}(t)$ modulo $\l^n$, for set of primes $\p$ of $F$ of Dirichlet density $1$.
\end{enumerate}
\end{cor} 

Note that, if in part $(ii)$, we instead just assume that the characteristic polynomial of $\rho_\l(\Frob_\p)$ is reducible modulo $\l^n$ for each $\p$, then the equivalence is false. Indeed, this latter condition is equivalent to $E$ being isogenous to an elliptic curve $E'$ that admits a cyclic $\l^n$-isogeny everywhere locally. However, there exist elliptic curves with no global isogenies that have isogenies everywhere locally \cites{Sutherland, anni, Banwait-Cremona, vogt}. We note that, since $\det\rho_\l$ is the $\l$-adic cyclotomic character, the assumption $\sqrt{\br{\frac{-1}\l}\l}\notin F$ imposed in \cite{Sutherland} and \cite{anni} forces $\rho_\l$ to be residually multiplicity-free.

\section{Preliminaries}

In this section, we recall key properties of the Bruhat--Tits tree $\X$ of $\PGL_2(K)$ and, for a representation $\rho\:G\to\GL_2(K)$, we define the $\rho(G)$-invariant subtree $\X(\rho)$ and discuss its shape. We then discuss the decompositions of $\rho$ modulo $\pi^n$ and define the index of irreducibility $n(\rho)$ and the index of irreducibility with multiplicity $m(\rho)$.
Our key reference is \cite{Bellaiche-arbres}*{\S2} (see also \cite{serre-trees}). 

\subsection{The \texorpdfstring{$\rho(G)$}{G}-invariant subtree}

Let $\rho\:G\map \GL_2(\fld{K})$ be a representation of a group $G$ and assume that there exists a $\rho(G)$-stable lattice $\Lambda\sub \fld{K}^2$. In particular, for every $g\in G$ the characteristic polynomial $P_{\rho(g)}(t)$ of $g$ is an element of $\O{}[t]$.
This condition is automatically satisfied for most representations of interest, for example, if $G$ is compact and $\rho$ is continuous.

Choosing a basis for $\Lambda$, we obtain a representation $\rho_\Lambda\:G\map \GL_2(\O{K})\sub \GL_2(K)$ that is isomorphic to $\rho$. 
Hence, we can define the residual representation $\orho_\Lambda\:G\to \GL_2(\F)$ as well as the mod $\pi^n$ representations $\rho_\Lambda\pmod{\pi^n}$. In general, the isomorphism class of the representation $\orho_\Lambda$ depends on the homothety class of $\Lambda$. 

\begin{defn}\label{defn_invariant-subtree}
    Let $A\sub \M_2(K)$ be a set of matrices.
    We denote by $\X(A)$ the induced subtree of $\X$ of all $A$-stable vertices, i.e.\ vertices $x\in \X$ for which $a\Lambda_x\sub \Lambda_x$ for all $a\in A$ and for some (and hence any) representative $\Lambda_x$.
    We define $\X(\rho)$ to be $\X(\rho(G))$.
\end{defn}

\begin{remarks}\mbox{}
\begin{enumerate}[leftmargin=*]
    \item If $A\sub A'$, then $\X(A)\supseteq \X(A')$.
    \item If $a,b\in \M_2(K)$ and $\lambda, \mu \in \O{}$, then $\X(\{a,b\}) = \X(\{a,b,ab\}) = \X(\{a,b, \lambda a + \mu b\})$. Hence, if $A\sub \M_2(K)$ is a finitely-generated $\O{}$-algebra, then we can compute $\X(A)$ by computing $\X(\{a_1, \ldots, a_n\})$ for an explicit set of generators. 
    \item Similarly, if $R\:\O{}[G]\to \M_2(K)$ is the algebra homomorphism corresponding to $\rho\:G\to \GL_2(K)$, then $\X(\rho) = \X(R(\O{}[G]))$.
\end{enumerate}
\end{remarks}

The following lemma shows the relationship between the structure of the graph $\X(\rho)$ and subrepresentations of $\rho_\Lambda\pmod{\pi^n}$:

\begin{lem}[\cite{Bellaiche-arbres}*{Prop.\ 11}]\label{lem:distance-submod-bijection}
    Let $x\in \X(A)$ and fix a representative $\Lambda_x$ of $x$. There is a bijection between
    \begin{itemize}
        \item Points $y\in \X(A)$ with $d(x, y) = n$;
        \item Free, rank one $\O{K}/\pi^n\O{K}$-submodules of $\Lambda_x/\pi^n\Lambda_x$ that are $A$-stable.
    \end{itemize}
Given a point $y\in\X(A)$ with $d(x,y)=n$, the corresponding submodule of $\Lambda_x/\pi^n\Lambda_x$ is given by $\Lambda_y/\pi^n\Lambda_x$, where $\Lambda_y$ is chosen so that $\pi^n\Lambda_x\subsetneq \Lambda_y\subsetneq\Lambda_x$.
\end{lem}

\begin{remark}\label{rem:bounded-irred}
Note that \Cref{lem:distance-submod-bijection} immediately shows that \Cref{thm:main-intro} is equivalent to \Cref{question1}. Moreover, we see that:
\begin{enumerate}
    \item The tree $\X(\rho)$ is bounded if and only if $\rho$ is an irreducible representation \cite{Bellaiche-arbres}*{Lem.\ 10}.
    \item  $\X(\rho)$ consists of a single point $x$ if and only if $\orho_{\Lambda_x}$ is irreducible.
    \item A vertex $x\in \X(\rho)$ is a leaf, i.e.\ a vertex with exactly one neighbour, if and only if $\Lambda_x/\pi\Lambda_x$ is indecomposable.
\end{enumerate}
\end{remark}

\subsection{The shape of \texorpdfstring{$\X(\rho)$}{the invariant building}}

In \cite{Bellaiche-arbres}, Bellaïche--Chenevier classify the possible shapes of $\X(\rho)$.
To describe this classification we recall their terminology.

\begin{defn}
    Let $S$ be a line segment in $\X$, and let $r$ be a positive integer. The \emph{band} $B(S,r)$ with \emph{nerve} $S$ and \emph{radius} $r$ is the subtree of $\X$ consisting of all vertices $x$ with $d(x, S)\le r$.
\end{defn}

We distinguish two particular types of bands:
\begin{itemize}[leftmargin=*]
    \item If $S=\braces{x}$ is a single vertex, $B(x,r):= B(\{x\}, r)$ is the ball of radius $r$ and centre $x$.
    \item If $S=\braces{x,y}$ consists of two adjacent vertices, we call $B(S,r)$ a \emph{generalised ball}. One can think of $B(S,r)$ as a ball of radius $r+\frac{1}{2}$ around the middle of the segment $[x,y]$.
\end{itemize}

   \begin{figure}[h]
       \centering
       \begin{tikzpicture}[
  grow cyclic,
  level distance=0.75cm,
  level 1/.style={sibling angle=90},
  level 2/.style={sibling angle=45},
  level 3/.style={sibling angle=30},
  nodes={circle,draw,inner sep=+0pt, minimum size=2pt},
  ]
\path[rotate=45]
  node[fill=black] {0}
  child foreach \cntI in {1,...,4} {
    node {}
    child foreach \cntII in {1,...,3} { 
      node {}
      child foreach \cntIII in {1,...,3} {
        node {}
      }
    }
  };
\end{tikzpicture}
\caption{A ball of radius $3$ in the Bruhat--Tits tree for $\Q_3$}
   \end{figure}

\begin{figure}[h]
       \centering
       \begin{tikzpicture}[
  grow cyclic,
  level distance=0.75cm,
  level 1/.style={sibling angle=60},
  level 2/.style={sibling angle=30},
  level 3/.style={sibling angle=15},
  nodes={circle,draw,inner sep=+0pt, minimum size=2pt},
  ]
\path[rotate=30]
  node[fill=black] {0}
    child[sibling angle=60] foreach \cntI in {1,...,3} {
    node {}
    child foreach \cntII in {1,...,3} { 
      node {}
      child foreach \cntIII in {1,...,3} {
        node {}
      }
    }
  }
  child[level distance=0.75cm, sibling angle=100] {node[fill=black]{0} 
  child[level distance=0.75cm, sibling angle=60] foreach \cntI in {1,...,3} {
    node {}
    child[sibling angle=30] foreach \cntII in {1,...,3} { 
      node {}
      child[sibling angle=15] foreach \cntIII in {1,...,3} {
        node {}
      }
    }
  }};
\end{tikzpicture}
\caption{A generalised ball of radius $3$ in the Bruhat--Tits tree for $\Q_3$}
   \end{figure} 
   
Bella\"iche--Chenevier have shown that if $\rho$ is irreducible, then $\X(\rho)$ is a band with finite diameter and radius \cite{Bellaiche-arbres}*{Thm.\ 21, Prop.\ 24}.

\begin{defn}
    Suppose that $\rho$ is irreducible, so that $\X(\rho)$ is a band. Define $d(\rho)$ to be the diameter of $\X(\rho)$ and $r(\rho)$ to be its radius.
\end{defn}

By definition, the diameter of a tree is the length of a maximal path. The diameter of the band $B(S,r)$ is therefore $\l(S) + 2r$, where $\l(S)$ is the length of the nerve $S$.

 \begin{figure}[h]
   \begin{tikzpicture}[
  grow cyclic,
  level distance=1cm,
  level 1/.style={sibling angle=90},
  level 2/.style={sibling angle=70},
  level 3/.style={sibling angle=30},
  level 3/.style={sibling angle=45},
  level 5/.style={sibling angle=15},
  nodes={circle,draw,inner sep=+0pt, minimum size=2pt},
  ]
\path[rotate=225]
  node[fill=black] {0}
   child foreach \cntI in {1,...,3} {
    node {}
    child[sibling angle=20] foreach \cntII in {1,...,3} {
    node {}
    }
    }
     child {node[fill=black]{1}
     child[sibling angle=90] {node{}
     child[sibling angle=20] foreach \cntI in {1,...,3} {
    node {}}}
     child {node[fill=black]{2} 
     child[sibling angle=90] {node{} 
     child[sibling angle=20] foreach \cntII in {1,...,3} {
    node {}
    }}
      child[sibling angle=90] {node[fill=black]{3}
      child[sibling angle=90] foreach \cntI in {1,...,3} {
    node {}
    child[sibling angle=20] foreach \cntII in {1,...,3} {
    node {}
    }}
      }
      child[sibling angle=90] {node{} 
     child[sibling angle=20] foreach \cntII in {1,...,3} {
    node {}
    }}
      } child[sibling angle=90] {node{} 
      child[sibling angle=20] foreach \cntI in {1,...,3} {
    node {}}
      }};
\end{tikzpicture}
\caption{A band in the Bruhat--Tits tree for $\Q_3$, with diameter $7$ and radius $2$. The solid points form the nerve $S$.}
\end{figure}

\newpage

To further illustrate the connection between the shape of $\X(\rho)$ and the representation theory of $\rho$, we record the following useful lemma:

\begin{lem}\label{lem:ball-implies-scalar}
Fix $A\sub \M_2(K)$ and let $H$ be the semigroup generated by $A$. Fix $x\in \X(A)$ and $r\ge 1$. Then $H$ acts by a one-dimensional character on $\Lambda_x/\pi^r\Lambda_x$ if and only if $B(x, r)\sub \X(A)$.
\end{lem}

\begin{proof}
    Note that $H$ acts as a one-dimensional character on $\Lambda_x/\pi^r\Lambda_x$ if and only if every free rank one $\O{K}/\pi^r\O{K}$-submodule of $\Lambda_x/\pi^r\Lambda_x$ is stable under the action of $H$. By \Cref{lem:distance-submod-bijection}, this latter condition is equivalent to $B(x, r)\sub\X(H)$. By \Cref{rem:bounded-irred}, $\X(A) = \X(H)$.
\end{proof}

\subsection{Conjugate characters modulo \texorpdfstring{$\pi^n$}{modulo powers of the uniformiser}}
Let $\rho\:G\to\GL_2(K)$ be a representation. 

\begin{defn}[c.f.\ \cite{Bellaiche-arbres}*{Def.\ 29}]\label{defn-n}
    Let $\chi_1,\chi_2\:G\to (\O{K}/\pi^n\O{K})\t$ be a pair of characters modulo $\pi^n$ for some positive integer $n$. We call $(\chi_1, \chi_2)$ a \emph{pair of conjugate characters for $\rho$ modulo} $\pi^n$ if the characteristic polynomial $P_{\rho(g)}(t)$ of $\rho(g)$ factors as 
    \[P_{\rho(g)}(t):=\det(t - \rho(g)) \equiv (t-\chi_1(g))(t-\chi_2(g)) \pmod{\pi^n}\]
    for all $g\in G$.
\end{defn}

\begin{remark}
    When the residue characteristic of $K$ is not $2$, by the identity $\trace(A)^2-\trace(A^2)=2\det(A)$ for all $A\in \M_2(K)$, the above condition is equivalent to
    \[\trace\rho(g) \equiv \chi_1(g) + \chi_2(g)\pmod{\pi^n}\]
    for all $g\in G$.
\end{remark}

\begin{remark}
    If $G$ is compact, $K$ is a $p$-adic local field, $\rho$ is continuous and $\chi\:G\to(\O{K}/\pi^n\O{K})\t$ is a character that factors through the image of $\rho$, then $\chi$ is automatically continuous. Indeed, any finite index subgroup of the closed subgroup $\rho(G)\sub \GL_2(\O{K})$ is open in $\rho(G)$ \cite{Lubotzky-Mann-powerful-p-groups}*{Thm.\ A}.
\end{remark}

\begin{defn}[c.f.\ \cite{Bellaiche-arbres}*{Def.\ 30}]\label{defn-m}
    If $(\chi_1,\chi_2)$ is a pair of characters modulo $\pi^n$, let $m(\chi_1, \chi_2)$ denote the largest integer $m\le n$ such that 
    \[\chi_1(g) \equiv\chi_2(g)\pmod{\pi^m}\]
    for all $g\in G$. 
\end{defn}

In particular, $m(\chi_1, \chi_2) = 0$ if $\chi_1\not\equiv\chi_2\pmod{\pi}$.

If $\rho$ is not residually multiplicity-free, then, in general, there exist multiple distinct pairs of conjugate characters modulo $\pi^n$ for every $n\ge 2$. The following lemma shows that any two pairs of conjugate characters agree modulo $\pi^s$, where $s$ is as in \Cref{thm:ribet-intro}.

\begin{lem}\label{lem:reordering}
Let $(\chi_1,\chi_2)$ and $(\eta_1,\eta_2)$ be two pairs of conjugate characters modulo $\pi^n$, for some $n$.
Let 
\[s = s(\chi_1, \chi_2) =\begin{cases}n-m(\chi_1, \chi_2) &\text{if }m(\chi_1,\chi_2)<\frac n2\\ \ceiling*{\frac n2} &\text{if } m(\chi_1,\chi_2)\ge\frac n2.\end{cases}\]
Then, up to reordering, we have
\[\chi_1\equiv\eta_1\pmod{\pi^s}\quad\text{and}\quad\chi_2\equiv\eta_2\pmod{\pi^s}.\]
\end{lem} 

\begin{proof}
    First suppose that $m(\chi_1, \chi_2) \ge \frac n2$. Let $i = 1$ or $2$. Then, for each $g\in G$, 
    \[v_\pi(\chi_1(g) - \chi_2(g))\ge \frac n2.\]
    It follows that $v_\pi(\eta_i(g) - \chi_1(g)) \ge \frac n2$ if and only if $v_\pi(\eta_i(g) - \chi_2(g)) \ge \frac n2$.

    Now, $\eta_i(g)$ is a root of $P_{\rho(g)}(t)\pmod{\pi^n}$, i.e.
    \[(\eta_i(g)-\chi_1(g))(\eta_i(g)-\chi_2(g))\equiv 0\pmod{\pi^n}.\]
    It follows that $\eta_i\equiv\chi_1\equiv\chi_2\pmod{\pi^{\ceiling{\frac n2}}}$. 
    
    Now suppose that $m(\chi_1, \chi_2) < \frac n2$. Choose $g\in G$ for which $v_\pi(\chi_1(g) - \chi_2(g)) = m(\chi_1, \chi_2)$. As before, we have
    \[(\eta_1(g)-\chi_1(g))(\eta_1(g)-\chi_2(g))\equiv 0\pmod{\pi^n}.\]
    Hence, at least one of the two factors must have valuation greater than or equal to $\frac{n}2$. Since ${m(\chi_1, \chi_2) <\frac n2}$ and since
    \[v_\pi\big((\eta_1(g)-\chi_2(g)) - (\eta_1(g)-\chi_1(g))\big) = v_\pi(\chi_1(g) - \chi_2(g)) = m(\chi_1, \chi_2),\]
    we see that one of the factors must have valuation exactly $m(\chi_1, \chi_2)$. Up to reordering $\chi_1$ and $\chi_2$, we may therefore assume that 
    \[v_\pi(\eta_1(g) - \chi_1(g)) \ge n - m(\chi_1, \chi_2)\quad\text{and}\quad v_\pi(\eta_1(g)-\chi_2(g)) = m(\chi_1, \chi_2).\]
    Let $h\in G$ and consider the two identities
       \begin{align*}
        \eta_1(h)-\chi_1(h)&=(\eta_1(g)-\chi_1(g))\eta_1(g^{-1}h)+\chi_1(g)(\eta_1(g^{-1}h)-\chi_1(g^{-1}h)),\\
        \eta_1(h)-\chi_2(h)&=(\eta_1(g)-\chi_2(g))\eta_1(g^{-1}h)+\chi_2(g)(\eta_1(g^{-1}h)-\chi_2(g^{-1}h)).
        \end{align*} 
    If $v_\pi(\eta_1(g^{-1}h)-\chi_1(g^{-1}h)) \ge n-m(\chi_1, \chi_2)$, it follows from the first identity that $\eta_1(h)\equiv\chi_1(h)\pmod{\pi^{n-m(\chi_1, \chi_2)}}$. Assume that $v_\pi(\eta_1(g^{-1}h)-\chi_1(g^{-1}h)) < n-m(\chi_1, \chi_2)$. Since $\eta_1(g\ii h)$ is a solution of the characteristic polynomial of $\rho(g\ii h)$, we have
    \[(\eta_1(g\ii h) - \chi_1(g\ii h))(\eta_1(g\ii h) - \chi_2(g\ii h)) \equiv 0\pmod{\pi^n}.\]
    Hence, $v_\pi(\eta_1(g\ii h) - \chi_2(g\ii h)) > m(\chi_1, \chi_2)$. Since $v_\pi(\eta_1(g) - \chi_2(g)) = m(\chi_1, \chi_2)$, it follows from the second identity that
    \[v_\pi(\eta_1(h) -\chi_2(h)) = m(\chi_1, \chi_2).\]
    Since $(\eta_1(h)-\chi_1(h))(\eta_1(h)-\chi_2(h))\equiv 0\pmod{\pi^n}$, we have $\eta_1(h) \equiv \chi_1(h) \pmod{\pi^{n-m(\chi_1, \chi_2))}}$ in this case as well.
    
    Repeating the above argument with $\eta_1$ replaced by $\eta_2$, we see that either $\chi_2\equiv \eta_1\pmod{\pi^{n-m(\chi_1,\chi_2)}}$, or $\chi_2\equiv \eta_2\pmod{\pi^{n-m(\chi_1,\chi_2)}}$.
    The first would imply that $\chi_1\equiv \chi_2\pmod{\pi^{n-m(\chi_1,\chi_2)}}$, contradicting the assumption that $m(\chi_1,\chi_2)<n/2$.
    Therefore, $\chi_2\equiv \eta_2\pmod{\pi^{n-m(\chi_1,\chi_2)}}$.
\end{proof}

\begin{remark}
    In fact, the proof of \Cref{lem:reordering} shows that, if $m(\chi_1, \chi_2) \ge \frac n2$, then $\chi_1\equiv\chi_2\equiv\eta_1\equiv\eta_2\pmod{\pi^s}$.
\end{remark}

\begin{cor}\label{cor:m-dichotomy}
Let $(\chi_1,\chi_2)$ and $(\eta_1,\eta_2)$ be two pairs of conjugate characters modulo $\pi^n$, for some $n$.
    \begin{enumerate}
    \item If $m(\chi_1,\chi_2)\geq n/2$, then $m(\eta_1,\eta_2)\geq n/2$.
    \item If $m(\chi_1,\chi_2)<n/2$, then $m(\eta_1,\eta_2)=m(\chi_1,\chi_2)$.
    \end{enumerate}
    
\end{cor} 

\begin{proof}
    If $m(\chi_1,\chi_2)\ge n/2$, then by \Cref{lem:reordering}, $\chi_i \equiv \eta_i\pmod{\pi^{\lceil\frac n2\rceil}}$ for each $i$. Hence,
    \[\eta_1 \equiv \chi_1\equiv\chi_2\equiv\eta_2\pmod{\pi^{\lceil\frac n2\rceil}}, \]
    so $m(\eta_1, \eta_2)\ge n/2$.

    If $m(\chi_1, \chi_2) < n/2$, then by \Cref{lem:reordering},
    \[\eta_i \equiv \chi_i\pmod{\pi^{n-m(\chi_1, \chi_2)}}\]
    for each $i$. Since $n-m(\chi_1, \chi_2) > m(\chi_1, \chi_2)$, it follows that $m(\chi_1, \chi_2) = m(\eta_1, \eta_2)$.
\end{proof}

\begin{remark}\label{rem:s-independent}
        In particular, the integer $s$ in the statement of \Cref{thm:ribet-intro} depends only on $\rho$ and $n$, and is independent of the choice of characters $\chi_1, \chi_2$.
\end{remark}

\subsection{The invariants \texorpdfstring{$n(\rho)$ and $m(\rho)$}{n and m}}
We recall the invariants $n(\rho)$ and $m(\rho)$ defined in \cite{Bellaiche-arbres}*{\S 4.2}.

\begin{defn}
Let $n(\rho)=\sup \braces{n : \text{there exists a pair of conjugate characters modulo } \pi^n}$.
\end{defn} 

\begin{defn}
If $n(\rho)$ is finite, then we define $m(\rho)=\max m(\chi_1,\chi_2)$,
where the maximum runs over all pairs of conjugate characters modulo $\pi^{n(\rho)}$.
\end{defn} 

We think of $n(\rho)$ as the \emph{index of reducibility} of $\rho$ and of $m(\rho)$ as the \emph{index of reducibility with multiplicity}.

An immediate consequence of \Cref{lem:distance-submod-bijection} is that $d(\rho)\leq n(\rho)$. The main result of \cite{Bellaiche-arbres} is that $d(\rho) \ge \max(n(\rho)/2, n(\rho) - m(\rho))$. Using \Cref{lem:distance-submod-bijection}, an immediate corollary of \Cref{thm:main-intro} is that $d(\rho) = n(\rho)$.
    
\begin{remark}\label{rem:n(rho)_vs_n(rho(G))}
    Define $n(\rho(G))$ to be the largest integer $n$ for which there exists a pair of conjugate characters $(\chi_1, \chi_2)$ modulo $\pi^{n}$, such that $\chi_1, \chi_2$ factor through $\rho(G)$. Then it is clear that $n(\rho(G))\le n(\rho)$. However, in general there do exist pairs of conjugate characters that do not factor through the image of $\rho$, so a priori, it is not clear that these invariants should be equal. However, it follows from \Cref{thm:main-intro} that indeed $n(\rho(G)) = n(\rho) = d(\rho)$. Note that in \cite{Bellaiche-arbres}, the authors only consider $n(\rho(G))$.
\end{remark}

\subsection{Boundary points}

\begin{defn}\label{defn_boundary-points}
Let $U\sub \X$ be a subset of the tree $\X$.
We say that a vertex $x\in U$ is an \emph{interior point} of $U$ if $B(x, 1)\sub U$. Otherwise, we say that $x$ is a \emph{boundary point} of $U$.
\end{defn} 

\begin{exmp}
    Suppose that $U=B(S,r)$ is a band with nerve $S$ and radius $r$.
    If $r=0$, then $U$ is a line segment and any vertex $x\in U$ is a boundary point.
    If $r>0$, the boundary points of $U$ are exactly its leaves.
\end{exmp} 

The following lemma is a slight generalisation of \cite{Bellaiche-arbres}*{Lem.\ 36}:

\begin{lem}\label{lem_existence_of_basis}
Let $g\in \M_{2}(\O{K})$ and let $x$ be a boundary point of $\X(\{g\})$.
Assume that the characteristic polynomial $P_g(t)$ of $g$ factors as
\[P_g(t)\equiv (t-\alpha) (t-\beta)\pmod{\pi^n},\]
for some $\alpha,\beta \in \O{K}/\pi^n\O{K}$.
Then there exists a basis $(v_1,v_2)$ of $\Lambda_x/\pi^n\Lambda_x$ with respect to which $g\pmod{\pi^n}$ is represented by the matrix
\[\twomat{\alpha}{1}{0}{\beta}.\]
\end{lem} 
    \begin{proof}
        Recall that, by \Cref{defn_invariant-subtree}, $\X(\{g\})$ is the subtree of $\X$ consisting of homothety classes $x$ of lattices $\Lambda_x$ such that $g\Lambda_x\sub\Lambda_x$. Note that, by \Cref{lem:ball-implies-scalar}, $x$ is an interior point of $\X(\{g\})$ if and only if $g$ acts by a scalar on $\Lambda_x/\pi\Lambda_x$
    
        Suppose that $x$ is a boundary point. Since $g$ does not act by a scalar on $\Lambda_x/\pi\Lambda_x$, there exists an element $\overline{v}\in \Lambda_x/\pi\Lambda_x$ such that $(\overline{v},g \overline{v})$ is a basis of $\Lambda_x/\pi\Lambda_x$.
    By Nakayama's lemma (for the ring $\O{K}$), if $v\in\Lambda_x/\pi^{n}\Lambda_x$ lies above $\overline{v}$, then $(v,g v)$ is basis of $\Lambda_x/\pi^n\Lambda_x$.
    Denote $w_2=v$ and let $w_1\in \Lambda_x/\pi^n\Lambda_x$ be the vector 
    \[w_1=(g-\beta)\cdot w_2.\]
    Clearly, $(w_1,w_2)$ is a basis of $\Lambda_x/\pi^n\Lambda_x$.
    In this basis, $g$ is represented by the matrix 
    \[\twomat{\alpha}{1}{0}{\beta}.\]
    Indeed, 
    \[g\cdot w_2 =  (g-\beta)\cdot w_2+\beta w_2 = w_1+\beta w_2\]
    and, by the Cayley--Hamilton theorem, 
    \[g\cdot w_1 = (g^2 - \beta g)\cdot w_2 = \alpha(g-\beta)\cdot w_2 = \alpha w_1.\]
        \end{proof} 

\begin{cor}\label{cor_d=n_for_a_single_A}
Let $g\in \M_2(\O{K})$, and let $x$ be a boundary point of $\X(\{g\})$.
Let $n$ be a positive integer.
The following are equivalent:
\begin{enumerate}
    \item The characteristic polynomial $P_g(t)$ is reducible modulo $\pi^n$.
    \item There exists a point $y\in\X(\{g\})$ with $d(x,y)=n$.
    \end{enumerate}
\end{cor} 

    \begin{proof}
    Assume $(i)$. 
    Let $(v_1,v_2)$ be a basis of $\Lambda_x/\pi^n\Lambda_x$ as in \Cref{lem_existence_of_basis}.
    Then $v_1$ spans a $g$-invariant rank one free submodule of $\Lambda_x/\pi^n\Lambda_x$, which by \Cref{lem:distance-submod-bijection} corresponds to a point $y\in \X(\{g\})$ at distance $n$ from $x$.
    
    Assume $(ii)$.
    Choose $\Lambda_y$ such that $\pi^n\Lambda_x\subsetneq \Lambda_y\subsetneq\Lambda_x$.
    The image of $\Lambda_y$ in $\Lambda_x/\pi^n\Lambda_x$ is a rank one free submodule, stable under the action of $g$. Thus, there exists a basis $(v_1,v_2)$ of $\Lambda_x/\pi^n\Lambda_x$ with respect to which $g$ is upper triangular. Hence, $P_g(t)$ is reducible modulo $\pi^n$.
    \end{proof} 

\section{Proof of Theorem $\ref{thm:main-intro}$}
Recall that $\rho\:G\map \mathrm{GL}_2(\fld{K})$ is a representation, with $G$ a compact group. Let $(\chi_1,\chi_2)$ be a pair of conjugate characters modulo $\pi^n$.
By definition, the characteristic polynomial $P_{\rho(g)}(t)$ of $\rho(g)$ factors as
\[P_{\rho(g)}(t)\equiv (t-\chi_1(g))(t-\chi_2(g))\pmod{\pi^n},\]
for all $g\in G$.

To prove \Cref{thm:main-intro}, instead of working with the group representation $\rho$, we work with the corresponding algebra homomorphism $R\:\O{}[G]\to \M_2(K)$. Write $\O{K}[\rho(G)] \sub \M_2(K)$ for the image of $R$. Then $\O{K}[\rho(G)]$ is the $\O{K}$-algebra spanned by the image of $G$ in $\GL_2(K)\sub\M_2(K)$. 

\begin{lem}\label{lem_extending_to_linear_span}
Let $g,h\in G$, let $a,b\in \O{K}$ and let $f=a \rho(g)+b \rho(h)\in \M_2(K)$.
Then, modulo $\pi^n$, the characteristic polynomial $P_f(t)$ of $f$ factors as
    \[
    P_f(t)\equiv \big(t-\sog{a \chi_1(g)+b\chi_1(h)}\big) \big(t-\sog{a \chi_2(g)+b\chi_2(h)}\big)\pmod{\pi^n}.  
    \]
\end{lem} 
    \begin{proof}
    We need to show that
        \begin{equation}\label{equation_trace}
        \trace(f)\equiv \sog{a \chi_1(g)+b\chi_1(h)}+\sog{a \chi_2(g)+b\chi_2(h)}\pmod{\pi^n}
        \end{equation} 
    and
        \begin{equation}\label{equation_determinant}
        \det(f)\equiv \sog{a \chi_1(g)+b\chi_1(h)}\sog{a \chi_2(g)+b\chi_2(h)}\pmod{\pi^n}.  
        \end{equation} 
    Identity $(\ref{equation_trace})$ follows from the linearity of the trace and the fact that $\trace(\rho(g)) \equiv \chi_1(g)+\chi_2(g)\pmod{\pi^n}$.
    Identity $(\ref{equation_determinant})$ follows from the identity
        \begin{equation}\label{equation_polarisation_identity}
        \det(A+B)-\det(A)-\det(B)=\trace(A)\cdot \trace(B)-\trace(A B),       
        \end{equation} 
    where $A,B\in \M_2(K)$, by taking $A=a\rho(g)$ and $B=b\rho(h)$.
    Indeed, we have
        \begin{align*}
        \det(f)
        &=\det(a\rho(g)+b\rho(h))\\
        &=\det(a\rho(g))+\det(b\rho(h))+\trace(a\rho(g))\cdot\trace(b\rho(h))-\trace((a\rho(g))\cdot (b\rho(h)))\\
        &=a^2\det(\rho(g))+b^2\det(\rho(h))+ab \trace(\rho(g)) \trace(\rho(h))-ab\trace(\rho(g)\rho(h))\\
        &\equiv a^2\chi_1(g)\chi_2(g)+b^2\chi_1(h)\chi_2(h)\\
        &\qquad +ab\big((\chi_1(g)+\chi_2(g))(\chi_1(h)+\chi_2(h))-(\chi_1(gh)+\chi_2(gh))\big)\pmod{\pi^n}\\
        &\equiv a^2 \chi_1(g)\chi_2(g)+b^2\chi_1(h)\chi_2(h)+ab\big(\chi_1(g)\chi_2(h)+\chi_1(h)\chi_2(g)\big)\pmod{\pi^n}\\
        &\equiv \sog{a\chi_1(g)+b\chi_1(h)}\sog{a\chi_2(g)+b\chi_2(h)}\pmod{\pi^n}.
        \end{align*} 
    \end{proof} 
    
    \begin{remark}
        The identity $(\ref{equation_polarisation_identity})$ is a polarisation identity. Indeed, consider the symmetric bilinear form $\langle A,B\rangle=\trace(A)\cdot\trace(B)-\trace(AB)$ on $\M_2(\fld{K})$.
        The associated quadratic form $q(A)=\langle A,A\rangle$ is equal to $2\det(A)$.
        As usual, one recovers a symmetric bilinear from its associated quadratic form by  $\langle A,B\rangle=\frac12(q(A+B)-q(A)-q(B))$, which is exactly identity $(\ref{equation_polarisation_identity})$.
    \end{remark} 

    Using the same calculation and induction, we deduce the following corollary:

    \begin{cor}\label{cor:extending-to-linear-span}
        Let $R\:\O{}[G]\to \M_2(K)$ be the algebra homomorphism corresponding to $\rho\:G\to\GL_2(K)$. Suppose that there is an integer $n\ge 1$ and $\O{}$-algebra homomorphisms $\chi_1, \chi_2\:\O{}[G]\to \O{}/\pi^n\O{}$ such that the characteristic polynomial $P_{\rho(g)}(t)$ of $\rho(g)$ factors as 
        \[P_{\rho(g)}(t) \equiv (t-\chi_1(g))(t-\chi_2(g))\pmod{\pi^n}\]
        for all $g\in G$. Then
        \[P_{R(g)}(t) \equiv (t-\chi_1(g))(t-\chi_2(g))\pmod{\pi^n}\]
        for all $g\in\O{}[G]$.
    \end{cor}

    We can now prove \Cref{thm:main-intro}:
    
    \begin{proof}[Proof of Theorem $\ref{thm:main-intro}$]
        If $\X(\rho) = \X$, then $\rho(G)$ contains only scalar matrices and the claim is easy.
        Hence, we may assume that $\X(\rho)$ has boundary points in the sense of \Cref{defn_boundary-points}.
                
        The fact that $(i)$ implies $(ii)$ follows immediately from \Cref{lem:distance-submod-bijection}. Indeed, if $x,y\in\X(\rho)$ are two points with $d(x,y)=n$, then there is a lattice $\Lambda_x$ such that $\Lambda_x/\pi^n\Lambda_x$ contains a free, rank one, $G$-stable $\O{K}/\pi^n\O{K}$-submodule. Let $\chi_1\:G\to (\O{K}/\pi^n\O{K})\t$ be the associated representation, and let $\chi_2$ be the quotient of $\rho_{\Lambda_x}\pmod{\pi^n}$ by $\chi_1$. Then $(\chi_1,\chi_2)$ is a pair of conjugate characters modulo $\pi^n$.
        
        Assume $(ii)$.
        Let $(\chi_1,\chi_2)$ be a pair of conjugate characters modulo $\pi^n$.
        Let $g_0\in G$ be an element for which $v_\pi(\chi_2(g_0)-\chi_1(g_0))$ is minimal.
        Here, $v_\pi$ is the truncated valuation on $\O{K}/\pi^n\O{K}$, taking values $0,1,\ldots,n$.

        We proceed differently in two cases according to whether $\X(\{\rho(g_0)\})=\X(\rho)$ or not.
        If $\X(\{\rho(g_0)\})=\X(\rho)$, then, by \Cref{cor_d=n_for_a_single_A} there exist $x,y\in\X(\{\rho(g_0)\})=\X(\rho)$ with $d(x,y)=n$, and we are done.
    
        Assume that $\X(\{\rho(g_0)\})\supsetneq \X(\rho)$.
        Then there exist two neighbours $x,y\in \X$ such that $x\in \X(\rho)$ and $y\in\X(\{\rho(g_0)\}) \setminus\X(\rho)$. In particular, there exists an element $g_1\in G$ such that $y\notin\X(\{\rho(g_1)\})$. By the choice of $g_0$, there exists an element $a\in\O{K}$ such that 
        \[\chi_2(g_1)-\chi_1(g_1)=a\sog{\chi_2(g_0)-\chi_1(g_0)}.\]
        Then
        \[\chi_2(g_1)-a\cdot\chi_2 (g_0)=\chi_1(g_1)-a\cdot\chi_1(g_0).\]
        Let $f=\rho(g_1)-a\cdot \rho(g_0)$, which is an element of $\O{K}[\rho(G)]$, and let $\delta=\chi_1(g_1)-a\cdot\chi_1(g_0)$.
        Then, by \Cref{lem_extending_to_linear_span},
        \[P_f(t)\equiv (t-\delta)^2\pmod{\pi^n}.\]
        By construction, the point $x$ is a boundary point of $\X(\{f\})$.
        Hence, by \Cref{lem_existence_of_basis}, there exists a basis $(v_1,v_2)$ of $\Lambda_x/\pi^n\Lambda_x$ with respect to which $f$ is represented by the matrix
        \[\twomat{\delta}{1}{0}{\delta}.\]
        Now, let $g\in G$ be any element and let 
        \[\twomat{a_g}{b_g}{c_g}{d_g}\in\GL_2(\O{K}/\pi^n\O{K})\]
        be the matrix that represents $\rho(g)$ with respect to the basis $(v_1,v_2)$.
        We will show that $c_g=0$ by computing $\trace(f\cdot\rho(g))\pmod{\pi^n}$ in two different ways.
        On the one hand,
            \begin{align*}
            \trace(f\cdot \rho(g))
            \equiv\trace\sog{\twomat{\delta}{1}{0}{\delta}\cdot\twomat{a_g}{b_g}{c_g}{d_g}}
            \equiv c_g+\delta\cdot \trace(\rho(g))\pmod{\pi^n}.
            \end{align*} 
       On the other hand,
            \begin{align*}
            \trace(f\cdot\rho(g))
            &=\trace((\rho(g_1)-a\cdot \rho(g_0))\cdot \rho(g))\\
            &=\trace(\rho(g_1g))-a\trace(\rho(g_0g))\\
            &\equiv \chi_1(g_1g)+\chi_2(g_1g)-a\cdot(\chi_1(g_0g)+\chi_2(g_0g))\pmod{\pi^n}\\
            &\equiv \sog{\chi_1(g_1)-a\cdot\chi_1(g_0)}\cdot\chi_1(g)+\sog{\chi_2(g_1)-a\cdot\chi_2(g_0)}\cdot\chi_2(g)\pmod{\pi^n}\\
            &\equiv \delta\cdot \chi_1(g)+\delta\cdot \chi_2(g)\pmod{\pi^n}\\
            &\equiv \delta\cdot\trace(\rho(g))\pmod{\pi^n}.
            \end{align*} 
        Combining the two computations, we see that $c_g\equiv 0\pmod{\pi^n}$ for all $g\in G$. Thus, the element $v_1$ generates a free $\rho(G)$-stable $\O{K}/\pi^n\O{K}$-submodule of $\Lambda_x/\pi^n\Lambda_x$.
        Hence, by \Cref{lem:distance-submod-bijection}, there exists a point $z\in \X(\rho)$ with $d(x, z) = n$.
        \end{proof}

    \begin{remark}
    In \Cref{thm:main-intro}, we can replace the assumption that $G$ is a group with the assumption that $G$ is a set with multiplication, and $\rho\:G\map \M_{2}(K)$ is a multiplicative map.
    Indeed, the notion of conjugate characters and the definition of $\X(\rho)$ apply as they are to the more general case, and the proof of \Cref{thm:main-intro} only uses the fact that $G$ is closed under multiplication.
    \end{remark} 

    The same arguments as above prove the following slightly more general statement in terms of $\O{}$-algebras:

        \begin{thm}\label{thm:main-alg}
        Let $A$ be an $\O{}$-algebra and let $R\:A\to \M_2(K)$ be an $\O{}$-algebra homomorphism. Assume that $A$ stabilises at least one lattice $\Lambda\sub K^2$.  Suppose that there is an integer $n\ge 1$ and $\O{}$-algebra homomorphisms $\chi_1, \chi_2\:A\to \O{}/\pi^n\O{}$ such that the characteristic polynomial $P_{R(g)}(t)$ of $R(g)$ factors as 
        \[P_{R(g)}(t) \equiv (t-\chi_1(g))(t-\chi_2(g))\pmod{\pi^n}\]
        for all $g\in A$. Then there exist $A$-stable lattices $\Lambda' \sub \Lambda\sub K^2$ such that $\Lambda'/\pi^n\Lambda\sub\Lambda/\pi^n\Lambda$ is a free $\O{}/\pi^n\O{}$-module of rank one.
    \end{thm}
    
\section{A generalisation of Ribet's Lemma}
\subsection{Proof of Theorem $\ref{thm:ribet-intro}$}
    Let $s=s(\chi_1,\chi_2)$. Since $\rho$ is irreducible, by \Cref{rem:bounded-irred}, the diameter $d(\rho)$ of $\X(\rho)$ is finite, and by \Cref{thm:main-intro} and \Cref{lem:distance-submod-bijection}, $d(\rho) = n(\rho)$. Let $x,y\in \X(\rho)$ be two vertices with $d(x, y) = n(\rho)$. 
    Then both $x$ and $y$ are leaves of $\X(\rho)$. 
    Choose lattices $\Lambda_x,\Lambda_y$ representing $x, y$ such that 
    \[\pi^{n(\rho)}\Lambda_x\subsetneq\Lambda_y\subsetneq\Lambda_x.\]
    By \Cref{lem:distance-submod-bijection}, $\Lambda_y/\pi^{n(\rho)}\Lambda_x$ is a free, rank 1, $G$-stable submodule of $\Lambda_x/\pi^{n(\rho)}\Lambda_x$, so $G$ acts on it by a character $\eta_1$.
    Similarly, $\pi^{n(\rho)}\Lambda_x/\pi^{n(\rho)}\Lambda_y$ is a $G$-stable, rank 1, free submodule of $\Lambda_y/\pi^{n(\rho)}\Lambda_y$, so $G$ acts on it by a character $\eta_2$.
    We see that $\Lambda_x/\pi^{n(\rho)}\Lambda_x$ is an extension of $\eta_2$ by $\eta_1$, and $\Lambda_y/\pi^{n(\rho)}\Lambda_y$ is an extension of $\eta_1$ by $\eta_2$.
    Both extensions are residually non-split because $\Lambda_x/\pi\Lambda_x$ and $\Lambda_y/\pi\Lambda_y$ are indecomposable, by \Cref{lem:distance-submod-bijection}.
    
    Denote by $\overline{\eta}_1$ and $\overline{\eta}_2$ the reductions of $\eta_1$ and $\eta_2$ modulo $\pi^s$.
    Then $\Lambda_x/\pi^{s}\Lambda_x$ is a residually non-split extension of $\overline{\eta}_2$ by $\overline{\eta}_1$ and $\Lambda_y/\pi^{s}\Lambda_y$ is a residually non-split extension of $\overline{\eta}_1$ by $\overline{\eta}_2$.
    By \Cref{lem:reordering} we can reorder $\eta_1,\eta_2$ so that $\chi_i\pmod{\pi^s}=\overline{\eta}_i$.\qed

\subsection{Linearly extendable characters and a counterexample}\label{sec:counterexample}

The version of Ribet's Lemma we give in \Cref{thm:ribet-intro} is optimal: if $(\chi_1, \chi_2)$ is a pair of conjugate characters modulo $\pi^n$, then it is not true in general that there exists a homothety class $x\in \X(\rho)$ such that $\Lambda_x/\pi^t\Lambda_x$ is a non-split extension of $\chi_1$ by $\chi_2$ when $t>s$. 
When $K = \Q_2$, one can already find counterexamples arising from elliptic curves over $\Q$.

\begin{exmp}
    Consider the elliptic curve $E=X_0(24)\: y^2 = x^3 -x^2 -4x + 4$, with LMFDB label 24.a4. Then $E[2](\Q) = \Z/4\Z \times \Z/2\Z$. If $\rho\:\Ga\Q\to \GL_2(\Q_2)$ is the $2$-adic Galois representation attached to $E$, then for all $g\in \Ga\Q$, we have $P_{\rho(g)}(t)\equiv(t- 1)(t-\chi_{cyc}(g))\pmod{8}$, where $\chi_{cyc}$ is the mod $8$ cyclotomic character. In this example, we have $n(\rho) = 3$, $m(\rho) = 1$, and $s(1, \chi_{cyc}) = 2$. 

    By \Cref{thm:ribet-intro}, there is a lattice for $\rho$ such that $\Lambda/4\Lambda$ is a non-split extension of $\chi_{cyc}$ by $1$ modulo $4$. Now, $\X(\rho)$ is exactly the $2$-power isogeny graph of $E$. The lattice $\Lambda$ should correspond to an isogenous elliptic curve $E'$ such that $E'[2](\Q) = \Z/4\Z$. One such choice is $E'\:y^2 = x^3- x^2 + x$, with LMFDB label 24.a5.
    
    Since $E$ is not $\Q$-isogenous to any elliptic curve with an $8$-torsion point, there is no stable lattice $\Lambda$ for $\rho$ such that $\Lambda/8\Lambda$ contains the trivial representation. Hence, \Cref{thm:ribet-intro} is optimal in this example. See \cite{Chiloyan}*{Table 2} for further counterexamples arising from elliptic curves over $\Q$.
\end{exmp}

More generally, when $K$ is an arbitrary complete valued field, there are two evident obstructions. 
First, if such an $x\in \X(\rho)$ exists, then $\chi_1$ and $\chi_2$ must factor through the image of $\rho$.
Second, the characters $\chi_1,\chi_2$ must extend $\O{K}$-linearly to $\O{K}[\rho(G)]$. Indeed, if there is a lattice $\Lambda_x$ such that $\Lambda_x/\pi^t\Lambda_x$ is a non-split extension of $\chi_1$ by $\chi_2$, then we can choose a basis $(v_1,v_2)$ of $\Lambda_x$ such that any $g\in G$ is represented by a matrix
\[\twomat{\widetilde{\chi}_1(g)}{b_g}{\pi^tc_g}{\widetilde{\chi}_2(g)},\]
where $\widetilde{\chi}_1(g),\widetilde{\chi}_2(g)$ are some lifts of $\chi_1(g),\chi_2(g)$ to $\O{K}$. Then, with respect to this basis, any element of $\O{K}[\rho(G)]$ is represented by a matrix in 
\[\M_0(\pi^t):=\set{\twomat{a}{b}{c}{d} : a,b,c,d\in\O{K},\ \pi^t\mid c}.\]
Since the maps $\M_0(\pi^t)\to \O{K}\pmod{\pi^t}$ given by $\br{\begin{smallmatrix}a & b\\c\pi^t&d\end{smallmatrix}}\mapsto a, d\pmod{\pi^t}$ are algebra homomorphisms, we see that $\chi_1, \chi_2$ extend $\O{K}$-linearly to  $\O{K}[\rho(G)]$.

\begin{defn}
Let $\chi$ be a character modulo $\pi^n$ of $G$.
We say that $\chi$ is linearly extendable with respect to $\rho$, or simply \emph{linearly extendable}, if $\chi$ factors through the image of $\rho$ and extends linearly to a map on $\O{K}[\rho(G)]$.
\end{defn} 

The character $\chi$ is linearly extendable if and only if for each $g_1,\ldots,g_k\in G$ and $a_1,\ldots,a_k\in \O{K}$ such that 
\[\sum_{i=1}^ka_i\rho(g_i)=0,\]
we have
\[\sum_{i=1}^ka_i\chi(g_i)=0.\]

Using these obstructions, we prove that \Cref{thm:ribet-intro} is best possible:

\begin{prop}
Fix integers $n, s$ such that $n\ge 1$ and $\frac n2\le s \le n$. Then there exists a group $G$, a representation $\rho\:G\map \GL_2(\O{K})$ with $n(\rho) = n$ and a pair of conjugate characters $(\chi_1,\chi_2)$ modulo $\pi^n$ with $s(\chi_1,\chi_2)  = s$ such that for all $t>s(\chi_1,\chi_2)$, there does not exist a vertex $x\in \X(\rho)$ such that $\Lambda_x/\pi^t\Lambda_x$ is a non-split extension of $\chi_1$ by $\chi_2$ or of $\chi_2$ by $\chi_1$.
\end{prop}

\begin{proof}
    Let 
    \[m = \begin{cases}
        n & s = \lceil \frac n2 \rceil\\
        n-s &\text{otherwise} 
    \end{cases}\]
    and let $G$ be the group
    \[G=\braces*{\twomat{a}{b}{\pi^nc}{d}\in\GL_2(\O{K})\ \big|\ a,b,c,d\in\O{K},\ a\equiv d\equiv 1\pmod{\pi^m}}\]
    and let $\rho\:G\hookrightarrow \GL_2(\O{K})$.
    By \cite{Lubotzky-Mann-powerful-p-groups}, any closed subgroup of $\GL_2(\O{K})$ is topologically finitely generated, so $G$ is topologically finitely generated.
    By the classification of topologically finitely generated abelian profinite groups, the abelianisation $G^{ab}$ of $G$ is isomorphic to $\Z_p^r\times T$ where $T$ is a finite group.
    
    We make the assumption that $r\geq 5$, which can always be achieved by choosing $\fld{K}$ appropriately. For example, for any prime $p$, we can take $K$ to be any degree $5$ extension of $\Qp$. Indeed, $G\supset\O{}\t$, embedded as the diagonal matrices, and $\O{}\t$ contains a finite index subgroup isomorphic to $\O{}^+\cong \Z_p^5$ as a $\Zp$-module.
    
    Let $\chi_1,\chi_2$ be the characters
    \[\chi_1\sog{\twomat{a}{b}{\pi^nc}{d}}=a\pmod{\pi^n},\ \ \ \chi_2\sog{\twomat{a}{b}{\pi^nc}{d}}=d\pmod{\pi^n}.\]
    Then $(\chi_1,\chi_2)$ is a pair of conjugate characters of $G$ modulo $\pi^n$.
    We have $n(\rho)=n$. Certainly $m(\chi_1,\chi_2)=m$, so $s(\chi_1, \chi_2) = s$ by definition. 
    The characters $\chi_1,\chi_2$ are evidently linearly extendable.
    We will construct another pair $(\eta_1,\eta_2)$ of conjugate characters of $G$ modulo $\pi^n$ such that $\eta_1\pmod{\pi^t}$ is linearly extendable only when $t\leq s(\eta_1, \eta_2)$. Note that, by \Cref{rem:s-independent}, $s(\eta_1, \eta_2) = s$.
    
    Let $g_1, \ldots, g_r, \ldots g_k\in G$ be the pre-images of a minimal set of topological generators for the abelianisation $G^{ab}\simeq \Z_p^r\times T$, such that the images of $g_1, \ldots, g_r$ generate $\Z_p^r$.  By assumption, $r\ge 5$, and $\M_2(\fld{K})$ has dimension $4$. Therefore, there exist elements $a_1, \ldots, a_r\in K$, not all zero, such that $\sum_{i = 1}^ra_ig_i = 0$ in $\M_2(\fld{K})$, and, up to scaling and reordering, we may assume that $a_i\in \O{K}$ for all $i$ and that $a_1 = 1$.
    
    Let $R_n$ denote the additive group of $\O{K}/\pi^n\O{K}$.
    Let $\gamma\in \pi^sR_n\setminus\pi^{s+1}R_n$.
    Let $\eps\:G\map R_n$ be the unique group homomorphism that satisfies
    \[\eps(g_1)=\chi_1(g_1)^{-1}\gamma,\ \ \ \eps(g_i)=0 \text{ for }i =2, \ldots, k.\]
    The image of $\eps$ also lies in $\pi^sR_n$, but not in $\pi^{s+1}R_n$.

    Define $\eta_1(g)=\chi_1(g)(1+\eps(g))$ and $\eta_2(g)=\chi_2(g)(1-\eps(g))$.
    A simple computation, based on the fact that $\eps(g)\eps(h)=0$, shows that $\eta_1$ and $\eta_2$ are characters.
    Next, since $\chi_1\eps=\chi_2\eps$ (recall that $2s\geq n$), it is easy to check that $\eta_1+\eta_2=\chi_1+\chi_2$, and since $\epsilon^2 =0$, we have $\eta_1\eta_2 = \chi_1\chi_2$
    Therefore, $\eta_1,\eta_2$ are conjugate characters modulo $\pi^n$.
    
    Finally, we show that $\eta_1\pmod{\pi^t}$ is linearly extendable only if $t\le s$.
    Recall that $\sum_{i=1}^ra_ig_i=0$ where $a_1,\ldots,a_r\in\O{K}$ and $a_1=1$.
    Using the fact that $\chi_1$ is linearly extendable, we see that, if $\eta_1\pmod{\pi^t}$ is linearly extendable, then
        \begin{align*}
        0\equiv\sum_{i=1}^ra_i\eta_1(g_i)
        \equiv\sum_{i=1}^ra_i\sog{\eta_1(g_i)-\chi_1(g_i)}
        \equiv\chi_1(g_1)\eps(g_1)
        \equiv\gamma\pmod{\pi^t}.
        \end{align*} 
    Since $\gamma\in \pi^sR_n\setminus\pi^{s+1}R_n$, we see that $\eta_1\pmod{\pi^t}$ is linearly extendable only if $t\leq s$.
    \end{proof}

\section{The shape of $\mathcal{X}(\rho)$}

In this section, we give a complete classification of the shape of $\X(\rho)$ in terms of the invariants $n(\rho)$ and $m(\rho)$. Conversely, given the shape of $\X(\rho)$ and one additional piece of data, we show how to compute $n(\rho)$ and $m(\rho)$. Given reasonable access to the representation $\rho$ (e.g.\ a finite list of topological generators of $\rho(G)$), it is easy to compute $\X(\rho)$, but comparatively difficult to compute $m(\rho)$ and $n(\rho)$.
Throughout this section we assume that the characteristic of the residue field of $K$ is not $2$.

After \cite{Bellaiche-arbres}*{Prop.\ 24}, we may assume, without loss of generality, that $\rho\:G\to\GL_2(K)$ is irreducible. In this case, $\X(\rho)$ is a band whose shape is determined by its radius $r(\rho)$ and its diameter $d(\rho)$ \cite{Bellaiche-arbres}*{Thm.\ 21}.

The following theorem, whose proof will occupy the rest of the paper, completes the classification of \cite{Bellaiche-arbres}*{Thm.\ 45}:

\begin{thm}\label{thm:classification1}
    Suppose that the residue characteristic of $K$ is not $2$. Let $\rho\:G\to\GL_2(K)$ be an irreducible representation. Then $\X(\rho)$ is a band of diameter $d(\rho) = n(\rho)$ and radius
    \[r(\rho)= \begin{cases} m(\rho) &\text{if } m(\rho)\ne n(\rho)\\ 
        \floor{\frac{m(\rho)}2} &\text{if } m(\rho) = n(\rho).
    \end{cases}\]
\end{thm}

Given $m(\rho)$ and $n(\rho)$, \Cref{thm:classification1} gives the exact shape of $\X(\rho)$. On the other hand, just knowing $d(\rho)$ and $r(\rho)$ is not enough to recover $n(\rho)$ and $m(\rho)$. By definition, if $x,y\in\X$ are neighbours, then the generalised ball $B(\{x,y\}, r)$ of radius $r$ is also a band of diameter $2r +1$ and radius $r$ and, in this case, two representations $\rho, \rho'$ with $n(\rho) = n(\rho')$ and $m(\rho) \ne m(\rho')$ can have $\X(\rho) = \X(\rho') = B(\{x,y\}, r)$.
We can tell these cases apart by using \emph{thin elements}.

\begin{defn}
Let $g\in G$.
We say that $g$ is a \emph{thin element} for $\rho$ if $\X(\{\rho(g)\})$ is an infinite band of radius $r(\rho)$.
Equivalently, $g$ is a thin element for $\rho$ if there exists a basis of $\fld{K}^2$ with respect to which $\rho(g)$ is represented by the matrix
\[\twomat{\alpha}{0}{0}{\beta},\]
where $\alpha,\beta\in\O{K}$ and $v_\pi(\beta-\alpha)=r(\rho)$.
\end{defn}

\begin{thm}\label{thm_classification}
Suppose that the residue characteristic of $K$ is not $2$. Let $\rho\:G\to\GL_2(K)$ be an irreducible representation, and suppose that $\X(\rho)$ is a finite band of diameter $d(\rho)$ and radius $r(\rho)$. Then $n(\rho) = d(\rho)$ and
\[m(\rho) = \begin{cases}
    d(\rho) &\text{ if }d(\rho) = 2r(\rho)\\
    d(\rho) &\text{ if }d(\rho) = 2r(\rho)+1\text{ and if $G$ does not contain a thin element.}\\
    r(\rho) &\text{otherwise}.
\end{cases}\]
\end{thm} 

\begin{remark}
    Suppose that $d(\rho) = 2r(\rho)+1$, i.e.\ that $\X(\rho)$ is a generalised ball. Then we can determine whether or not $G$ contains a thin element by looking at any list of topological generators of $\rho(G)$, since, in this case, if $g,h\in G$ and $gh$ is a thin element, then $g$ or $h$ must be a thin element.
    
    Indeed, let $(\chi_1,\chi_2)$ be a pair of conjugate characters modulo $\pi^{d(\rho)}$ of $G$.
    First, since $d(\rho)>2r(\rho)$, by a version of Hensel's lemma (\Cref{lem_Hensel}), if $v_\pi(\chi_2(g)-\chi_1(g))\leq r(\rho)$, for an element $g\in G$, then $g$ is a thin element.
    Thus, if $h$ and $g$ are not thin elements
    \[v_\pi(\chi_2(g)-\chi_1(g))>r(\rho),\ \ \ \text{and}\ \ \ v_\pi(\chi_2(h)-\chi_1(h))>r(\rho).\]
    By writing
    \[\chi_2(gh)-\chi_1(gh)
    =\chi_2(g)(\chi_2(h)-\chi_1(h))+\chi_1(h)(\chi_2(g)-\chi_1(g)),
    \]
    we see that 
    \[v_\pi\sog{\chi_2(gh)-\chi_1(gh)}>r(\rho).\]
    Thus, $gh$ is not a thin element of $G$.
    
    Therefore, if $d(\rho) = 2r(\rho)+1$, then $G$ contains a thin element if and only if one the generators of $\rho(G)$ is a thin element.
\end{remark}

\subsection{The invariant \texorpdfstring{$k(\rho)$}{k}}

\begin{defn}
We denote by $k(\rho)$ the largest integer $k$ for which the map
\[g\mapsto \frac{\trace\rho(g)}{2}\pmod{\pi^k}\]
is multiplicative.
\end{defn} 

Let $\chi$ denote the character $g\mapsto \frac{\trace\rho(g)}{2}\pmod{\pi^k}$. Then, by definition, $\trace(\rho) = \chi+\chi\pmod{\pi^k}$, and since $p\neq 2$, $(\chi,\chi)$ is a pair of conjugate characters modulo $\pi^{k(\rho)}$. 
In particular, 
\[P_{\rho(g)}(t)\equiv \sog{t-\frac{\trace\rho(g)}{2}}^2\pmod{\pi^{k(\rho)}}\]
for all $g\in G$.
Note that $m(\rho)\le k(\rho) \le n(\rho)$.

\begin{remark}
    Recall that the invariant $m(\rho)$ is the largest integer $m$ for which there exists a pair of conjugate characters $(\chi_1, \chi_2)$ modulo $\pi^{n(\rho)}$ such that $\chi_1\equiv \chi_2\pmod{\pi^m}$. In particular, $m(\rho)$ is the largest integer $m$ for which there exists a character $\chi$ modulo $\pi^m$ with $\trace\rho(g) = 2\chi\pmod{\pi^m}$ such that $\chi$ lifts to a character modulo $\pi^{n(\rho)}$. By contrast, $k(\rho)$ is the largest integer $k$ for which there exists a character $\chi$ modulo $\pi^k$ such that $\trace\rho = 2\chi\pmod{\pi^k}$, but $\chi$ need not lift to a character modulo $\pi^{n(\rho)}$.
\end{remark}

\begin{remark}
    There appears to be a gap in the proof of \cite{Bellaiche-arbres}*{Thm.\ 45}: in \cite{Bellaiche-arbres}*{Prop. 32}, the authors claim that $r(\rho) \le m(\rho)$. However, their proof only shows that $r(\rho) \le k(\rho)$. Indeed, their proof shows that there is a lattice $\Lambda$ such that the action of $G$ on $\Lambda/\pi^{r(\rho)}\Lambda$ is by a character, however, it is not clear why this character should extend to a character on $\Lambda/\pi^{n(\rho)}\Lambda$. The fact that $r(\rho) \le m(\rho)$ is a consequence of the next three lemmas.
\end{remark}

\begin{lem}[c.f.\ \cite{Bellaiche-arbres}*{Prop.\ 31}]\label{prop_k(G)_geq_2m(G)}
\begin{enumerate}
    \item Either $m(\rho)=n(\rho)$ or $2m(\rho)<n(\rho)$.
    \item We have
    \[k(\rho) = \begin{cases}2m(\rho) &\text{if } m(\rho)\ne n(\rho)\\
m(\rho) &\text{if }m(\rho) = n(\rho).\end{cases}\]
\end{enumerate}

\end{lem} 
    \begin{proof}
    Part $(i)$ is \cite{Bellaiche-arbres}*{Prop.\ 31}.
    
    For part $(ii)$, if $m(\rho) = n(\rho)$, the result follows immediately from the inequality $m(\rho)\leq k(\rho)\leq n(\rho)$. So assume that $m(\rho) \ne n(\rho)$, and let $(\chi_1, \chi_2)$ be a pair of conjugate characters modulo $\pi^{n(\rho)}$ such that $\chi_1\equiv\chi_2\pmod{\pi^{m(\rho)}}$. Since 
    $\chi_1(g)\equiv\chi_2(g)\pmod{\pi^{m(\rho)}}$,
    it is easy to verify that the map
    \[g\mapsto \frac{\chi_1(g)+\chi_2(g)}{2}=\frac{\trace\rho(g)}{2}\pmod{\pi^{2m(\rho)}}\]
    is multiplicative. Hence, $2m(\rho) \le k(\rho)$. 

    For $i=1,2$ we have
    \[\sog{\chi_i(g)-\frac{\trace(\rho(g))}{2}}^2\equiv 0 \pmod{\pi^{k(\rho)}},\]
    for all $g\in G$.
    Therefore, $\chi_1(g)\equiv \frac{\trace(\rho(g))}{2}\equiv \chi_2(g)\pmod{\pi^{\ceiling{k(\rho)/2}}}$, for all $g\in G$.
    It follows that $2m(\rho)\geq k(\rho)$. Hence, $2m(\rho) = k(\rho)$.
    \end{proof} 

\begin{lem}\label{prop_k(G)_and_2r(G)}
 We have $k(\rho)\geq 2r(\rho)$.
 If, in addition, $G$ contains a thin element then $k(\rho)=2r(\rho)$.
 \end{lem} 
    \begin{proof}
    Let $x$ be a point on the nerve of $\X(\rho)$ and let $\l_1,\l_2$ be two leaves of $\X(\rho)$, both at distance $r(\rho)$ from $x$ and at distance $2r(\rho)$ from one another.
    Then there exists a basis $(v_1,v_2)$ of $\Lambda_{\l_1}$ such that $(v_1,\pi^{r(\rho)}v_2)$ is a basis of $\Lambda_x$ and $(v_1,\pi^{2r(\rho)}v_2)$ is a basis of $\Lambda_{\l_2}$.
    
    For all $g\in G$, let 
    \[\twomat{a_g}{b_g}{\pi^{2r(\rho)}c_g}{d_g}\]
    be the matrix representing the action of $g$ with respect to the basis $(v_1,v_2)$. Then the action of $g$, with respect to the basis $(v_1,\pi^{r(\rho)}v_2)$, is 
    \[\twomat{a_g}{\pi^{r(\rho)}b_g}{\pi^{r(\rho)}c_g}{d_g}.\]
   
    From the first presentation, we see that the maps $g\mapsto a_g\pmod{\pi^{2r(\rho)}}$ and $g\mapsto d_g\pmod{\pi^{2r(\rho)}}$ are characters of $G$. Moreover, since $x$ is in the nerve of $\X(\rho)$, by definition, $B(x, r(\rho))\sub \X(\rho)$. Hence, by \Cref{lem:ball-implies-scalar}, $G$ acts on $\Lambda_x/\pi^{r(\rho)}\Lambda_x$ by a character. Therefore, $a_g\equiv d_g\pmod{\pi^{r(\rho)}}$. It follows that the function $G\map (\O{K}/\pi^{2r(\rho)}\O{K})^\times$ given by
    \[g\mapsto \frac{1}{2}(a_g+d_g)\equiv \frac{\trace(g)}{2}\pmod{\pi^{2r(\rho)}}\]
    is a multiplicative character, and hence $k(\rho)\geq 2r(\rho)$.
    
    Assume that $G$ contains a thin element $h\in G$.
    It remains to show that $k(\rho)\leq 2r(\rho)$.
    The characteristic polynomial of $h$ factors as
    $P_h(t)=(t-\alpha)(t-\beta)$,
    for some $\alpha,\beta\in\O{K}$ with $v_\pi(\beta-\alpha)=r(\rho)$.
    Modulo $\pi^{k(\rho)}$ we have
    \[P_h(t)\equiv (t-\alpha)(t-\beta)\equiv (t-\gamma)^2\pmod{\pi^{k(\rho)}},\]
    where $\gamma=\frac12\trace(h) = \frac{\alpha+\beta}2$. It follows that 
    \[\br{\frac{\alpha+\beta}2}^2 \equiv \alpha\beta\pmod{\pi^{k(\rho)}}\]
    and hence that $(\alpha-\beta)^2 \equiv 0\pmod{\pi^{k(\rho)}}$. Hence $k(\rho)\leq 2r(\rho)$.
    \end{proof} 

\subsection{Proof of Theorems $\ref{thm:classification1}$ and $\ref{thm_classification}$}

Recall the following version of Hensel's Lemma \cite{Eisenbud-CommutativeAlgebra}*{Thm.\ 7.3}:

\begin{lem}\label{lem_Hensel}
Let $P(t)\in \O{K}[t]$ be a monic quadratic polynomial and let $n$ be a positive integer. Suppose that
\[P(t)\equiv (t-\alpha)(t-\beta)\pmod{\pi^n},\]
for some $\alpha,\beta\in\O{K}$ with $v_\pi(\alpha-\beta)<\frac n2$.
Then there exist $\alpha',\beta'\in \O{K}$ such that 
\[P(t)=(t-\alpha')(t-\beta').\]
Moreover, $v_\pi(\alpha-\beta)=v_\pi(\alpha'-\beta')$.
\end{lem} 

\begin{lem}\label{prop_thin_element}
If $d(\rho) > 2r(\rho)+1$, then $G$ contains a thin element.
\end{lem}

\begin{proof}
    Fix $x,y\in\X(\rho)$ with $d(x,y)=d(\rho)$. Then there exists a basis $(v_1,v_2)$ of $\Lambda_x$ such that $(v_1,\pi^{d(\rho)}v_2)$ is a basis of $\Lambda_y$. With respect to this basis, the action of any $g\in G$ is represented by a matrix of the form
    \[\twomat{a_g}{b_g}{\pi^{d(\rho)}c_g}{d_g}.\]
    Now let $z\in\X(\rho)$ be a vertex in the centre of the nerve of $\X(\rho)$ (if the nerve has even cardinality, let $z$ be the central point closest to $x$). Let $\l\in\X(\rho)$ be a boundary point with $d(\l,z)=r(\rho)$.
    If $\X(\rho)=B(S,r(\rho))$ and $S$ is a line segment of length $a$, then $d(\rho)=a+2r(\rho)$ and $d(x,z)=r+\lfloor a/2\rfloor$.
    Note that if $r(\rho)=0$, then $\l=z$.
     \begin{center}
    \includegraphics[width=12cm]{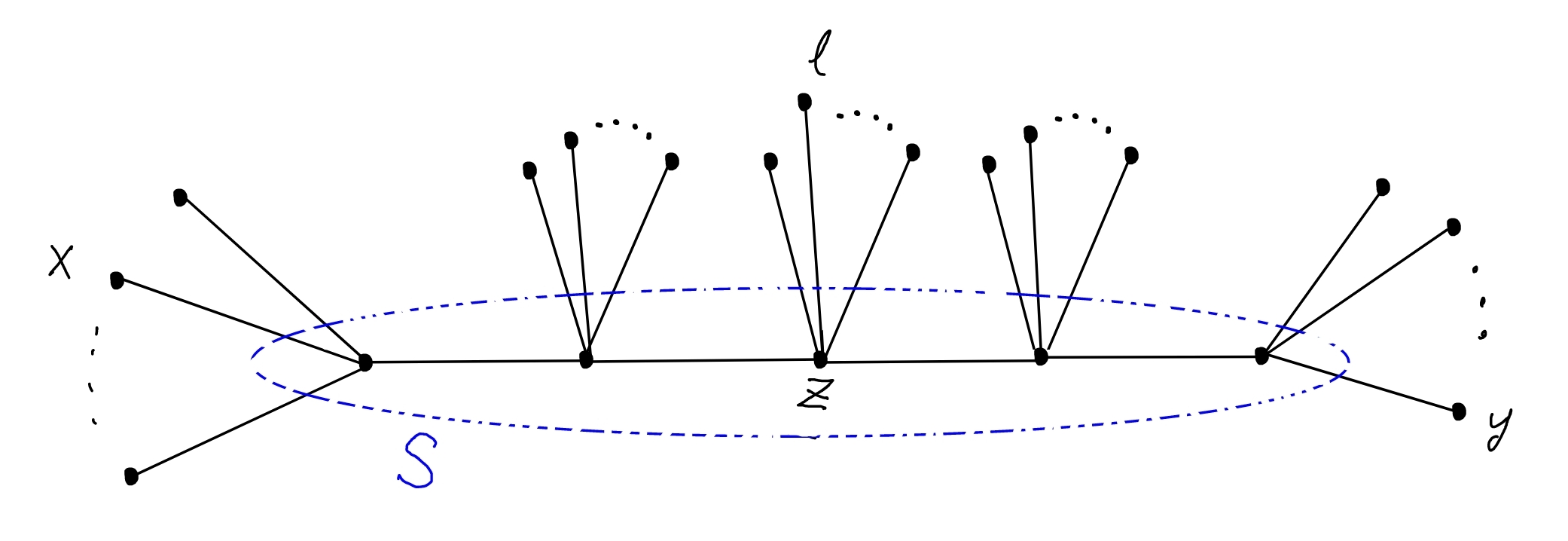}
    \end{center}
    We can choose an element $h\in G$ such that both $x$ and $\l$ are boundary points of $\X(\{h\})$. Indeed, let $g_1\in G$ be such that $x$ is a boundary point of $\X(\{\rho(g_1)\})$ and let $g_2\in G$ be such that $\l$ is a boundary point of $\X(\{\rho(g_2)\})$. Then either $\l$ is also a boundary point of $\X(\{\rho(g_1)\})$, $x$ is a boundary point of $\X(\{\rho(g_2)\})$ or both $x$ and $\l$ are boundary points of $\X(\{\rho(g_1g_2)\})$. Indeed, if $x$ is not a boundary point of $\X(\{\rho(g_2)\})$, then, by \Cref{defn_boundary-points}, all its neighbours are fixed by $\rho(g_2)$. Since $x$ is a boundary point of $\X(\{\rho(g_1)\})$ by definition, there exists a neighbour $x'$ of $x$ such that $\rho(g_1)\Lambda_{x'} \ne \Lambda_{x'}$, where $\Lambda_{x'}$ is any representative of $x'$. But then $\rho(g_1g_2)\Lambda_{x'} = \rho(g_1)\Lambda_{x'} \ne \Lambda_{x'}$, so $x$ is not a boundary point of $\X(\{\rho(g_1g_2)\})$. The argument for $\l$ is similar. 
    
    The vectors $(v_1,\pi^{d(x,z)}v_2)$ form a basis of $\Lambda_z$, with respect to which, the action of $h$ is represented by the matrix
    \[\twomat{a_h}{\pi^{d(x,z)}b_h}{\pi^{d(\rho)-d(x,z)}c_h}{d_h}.\]
    Note that, by the definition of $z$, $r(\rho)< \frac{d(\rho)-1}2\le  d(x,z)$ and that $r(\rho)<d(\rho)-d(x,z)$. 
    Therefore, modulo $\pi^{r(\rho)+1}$, $h$ is represented by the matrix
    \[\twomat{a_h}{0}{0}{d_h}\pmod{\pi^{r(\rho)+1}}.\]
    Now, $\X(\rho)$ contains the ball of radius $r(\rho)$ around $z$, so $v_\pi(d_h - a_h) \ge r(\rho)$. On the other hand, by the way we chose $h$, $B(z,r(\rho)+1)\not\sub \X(\{h\})$, so, by \Cref{lem:ball-implies-scalar}, $h$ does not act as a scalar on $\Lambda_z/\pi^{r(\rho) + 1}\Lambda_z$. Hence, $v_\pi(d_h - a_h) = r(\rho)$. By \Cref{lem_Hensel} we see that $h$ is a thin element for $\rho$.
\end{proof}

\begin{proof}[Proof of Theorems $\ref{thm:classification1}$ and $\ref{thm_classification}$]
    By \Cref{thm:main-intro}, we always have $d(\rho) = n(\rho)$. Note that \Cref{thm_classification} implies \Cref{thm:classification1}, so we only need to prove the former theorem.
    
    If $d(\rho) = 2r(\rho)$, then, by \Cref{prop_k(G)_and_2r(G)}
    \[n(\rho) = d(\rho) = 2r(\rho) \le k(\rho) \le n(\rho)\]
    so $k(\rho) = n(\rho)$. Hence, by \Cref{prop_k(G)_geq_2m(G)}, $m(\rho) = n(\rho)$.
    
    If $d(\rho)>2r(\rho) + 1$, then, by \Cref{prop_thin_element}, $G$ contains a thin element, so by \Cref{prop_k(G)_and_2r(G)},
    \[k(\rho) = 2r(\rho) < d(\rho) = n(\rho),\]
    so $k(\rho) \ne n(\rho)$. Hence, by \Cref{prop_k(G)_geq_2m(G)}, $k(\rho) = 2m(\rho)$. So $m(\rho) = r(\rho)$.
    
    Finally, suppose that $d(\rho)=2r(\rho) + 1$. If $G$ contains a thin element, then, by \Cref{prop_k(G)_and_2r(G)}, $k(\rho) = 2r(\rho) = d(\rho) -1 < n(\rho)$. Thus, by \Cref{prop_k(G)_geq_2m(G)}, $k(\rho) = 2m(\rho)$, so $m(\rho) = r(\rho)$.
    
    Now suppose that $G$ does not contain a thin element. Suppose for contradiction that $m(\rho) \ne n(\rho)$. Then, by \Cref{prop_k(G)_geq_2m(G)}, $2m(\rho)<n(\rho)$. Thus, there exists an element $g\in G$ such that the characteristic polynomial of $\rho(g)$ factors as 
    \[P_{\rho(g)}(t)\equiv (t-\alpha) (t-\beta)\pmod{\pi^{n(\rho)}},\]
    with $\alpha,\beta\in\O{K}/\pi^{n(\rho)}\O{K}$ satisfying $v_\pi(\beta-\alpha)=m(\rho)<n(\rho)/2$.
    Thus, by \Cref{lem_Hensel}, $g$ is a thin element in $G$, a contradiction.
\end{proof}

\vspace{-10pt}
\section*{Acknowledgements}

We'd like to thank Ehud de Shalit for introducing us to this area of study, and for his advice and feedback throughout this project. We are grateful to Ken Ribet for pointing out to us the relevance of the methods of \cite{KatzAbelian1980} to our work.
We are also grateful to Ga\"etan Chenevier, John Cullinan and Alex Lubotsky for helpful conversations. We'd like to thank the referees for their detailed comments and corrections, which have greatly improved the quality of this paper.
The second author was supported by an Emily Erskine Endowment Fund postdoctoral fellowship at the Hebrew University of Jerusalem, by the Israel
Science Foundation (grant No. 1963/20) and by the US-Israel Binational Science Foundation (grant No. 2018250).

\vspace{-5pt}
\bibliography{bibliography}
\bibliographystyle{alpha}
\end{document}